\documentclass[11pt,a4paper,twoside]{article}
\usepackage{amsthm, amsfonts,amsmath, color}

\newtheorem{theorem}{Theorem}
\newtheorem{corollary}{Corollary}
\newtheorem{definition}{Definition}
\newtheorem{example}{Example}
\newtheorem{lemma}{Lemma}
\newtheorem{proposition}{Proposition}

\def\Ric{\operatorname{Ric}}
\def\tr{{\rm\,trace\,}}

\title{$\ast$-$\eta$-Ricci solitons and Einstein metrics \\ on a weak $\beta$-Kenmotsu manifold}
\author{Vladimir Rovenski
\footnote{Department of Mathematics, University of Haifa, Mount Carmel, 3498838 Haifa, Israel
\newline e-mail: {\tt vrovenski@univ.haifa.ac.il}
}}

\date{}

\begin{document}

\maketitle

\begin{abstract}
Weak almost contact metric manifolds (i.e., the complex structure is replaced by a nonsingular skew-symmetric tensor), defined by the author and R.\,Wolak, allow a new look at the classical theory and find novel applications.
An important case of these manifolds, which is locally a twisted product, is a weak $\beta$-Kenmotsu manifold defined by the author and D.S.\,Patra.
In the paper, the concept of the $\ast$-Ricci tensor is adapted to weak almost contact manifolds,
the interaction of the $\ast$-$\eta$-Ricci soliton with the weak $\beta$-Kenmotsu structure (with $\beta=const$) is studied and new characteristics of Einstein metrics are~obtained.

\vskip1mm\noindent
\textbf{Keywords}:
weak almost contact structure,
warped product,
weak $\beta$-Kenmotsu manifold,
$\ast$-Ricci tensor,
$\ast$-$\eta$-Ricci soliton,
Einstein metric

\vskip1mm\noindent
\textbf{Mathematics Subject Classifications (2010)} Primary 53C12; Secondary 53C21
\end{abstract}

\section{Introduction}

Solutions of some non-linear  partial differential equations decompose at large time $t$ into solitary waves running with constant speed -- the so-called solitons.
Ricci solitons (\textbf{RS}), being self-similar solutions of the Ricci flow $\partial g/\partial t = -2\Ric_g$ (a parabolic partial differential equation),
generalize Einstein metrics $g$, i.e., solutions of $\Ric = \lambda\,g$, where $\lambda$ is a real constant, see \cite{CLN-2006}.
A~Riemannian~metric $g$ on a smooth manifold $M$ is called a RS if there exists a smooth vector field $V$, called soliton vector field, and a scalar $\lambda$, called soliton constant (some authors assume $\lambda$ to be a smooth function on $M$), satisfying
\begin{align}\label{E-sol-basic}
 (1/2)\pounds_V\,g + \Ric = \lambda\,g,
\end{align}
where $\pounds$ is the Lie derivative. A RS is {expanding}, {steady}, or {shrinking} if $\,\lambda$ is negative, zero, or positive, respectively.
If $\pounds_V\,g=0$, that is, $V$ is zero or a Killing vector field, then the RS reduces to Einstein metric.
The study of RS is also motivated by the fact that some compact manifolds don't admit Einstein metrics.
Some authors study the generalization of \eqref{E-sol-basic} to $\eta$-RS (where $\eta$ is a 1-form) defined in \cite{cho2009ricci}~by
\begin{align}\label{E-sol-eta}
 (1/2)\pounds_V\,g+\Ric=\lambda\,g +\mu\,\eta\otimes\eta,
\end{align}
where $\mu$ is a smooth function on $M$. If $\pounds_V\,g=0$, then \eqref{E-sol-eta} gives $\eta$-Einstein metrics
\begin{align}\label{E-eta-Einstein}
 \Ric=\lambda\,g +\mu\,\eta\otimes\eta.
\end{align}

An interesting question is how Ricci-type solitons interact with various geometric structures, e.g. contact structure or foliation, on a manifold.
When does an almost contact metric manifold equipped with a Ricci-type soliton carry canonical (e.g. Einstein-type) metrics?

Contact Riemannian geometry is of growing interest due to its important role in physics, e.g.,~\cite{blair2010riemannian}.
An~{almost contact structure on a Riemannian manifold $(M^{2n+1},g)$ is a set $(f,\xi,\eta)$, where
${f}$ is a $(1,1)$-tensor of rank $2\,n$, $\xi$ is a unit vector field, ${\eta}$ is a 1-form dual to $\xi$, satisfying the equality
$g({f} X,{f} Y)= g(X,Y) - {\eta}(X)\,{\eta}(Y)$,
In 1969, S.\,Tanno classified connected almost contact metric mani\-folds with automorphism groups of maximal dimensions, as follows:
(i)~homogeneous normal contact Riemannian manifolds with constant $f$-holomorphic sectional curvature if the $\xi$-sectional curvature $K(\xi, X) > 0$;
(ii)~warped products of the real line and a K\"{a}hler manifold, if $K(\xi, X)>0$;
(iii)~metric products if $K(\xi, X)=0$.
Warped products are also important for  theoretical physics: some of Einstein metrics and spacetime models are of this~type.
In \cite{olszak1991normal}, Z.~Olszak cha\-racterized a class (ii), known as $\beta$-Kenmotsu manifolds ($\beta$-\textbf{KM}) for $\beta=const\ne0$
(defined by K.~Kenmotsu in \cite{kenmotsu1972class}, when $\beta=1$), by the~equality
\begin{equation}\label{2.3-patraA}
 (\nabla_{X} f)Y=\beta\{g(f X, Y)\,\xi -\eta(Y)\,f X\} .
\end{equation}
If $\beta\ne0$ is a smooth function on the manifold, then \eqref{2.3-patraA} gives a twisted product, see Section~\ref{sec:02} and \cite{rov-119}.

\smallskip

The study of RS and their analogues on almost contact metric manifolds,
particularly, on almost Sasakian, almost cosymplectic and almost KM, has gained great popularity,
e.g. \cite{Gosh-Patra2018,kenmotsu1972class,Li-2022,Mondal-2024,Patra-2024,ven-2019}.
There has been interest in the so-called $\ast$-RS, which is defined by
\begin{align*}
 (1/2)\,\pounds_V\,g + \Ric^\ast = \lambda\,g .
\end{align*}
The $\ast$-Ricci tensor (defined by S.~Tashibana \cite{Tash-1959} on almost Hermitian manifolds and then applied by T.~Hamada \cite{Hamada-2002} to almost contact metric manifolds) is given by
\begin{align}\label{E-ast-Ricci-tensor}
 \Ric^\ast(X,Y) = (1/2)\tr\{Z\to R_{X, fY} fZ\} ,
\end{align}
where
$R_{{X},{Y}}Z=\nabla_X\nabla_Y Z-\nabla_Y\nabla_X Z-\nabla_{[X,Y]} Z$ is the curvature tensor.

Inspired by the works about $\ast$-RS, first S.~Dey and S.~Roy \cite{DeyRoy-2020} for $V=\xi$,
and then S.~Dey, S.~Sarkar and A.~Bhattacharyya \cite{DSB-2024} for arbitrary $V$, defined $\ast$-$\eta$-RS on almost contact metric manifolds by replacing the Ricci tensor in the equation \eqref{E-sol-eta} with the $\ast$-Ricci~tensor:
\begin{align}\label{E-sol-ast-eta0}
 (1/2)\,\pounds_V\,g+\Ric^\ast=\lambda\,g +\mu\,\eta\otimes\eta.
\end{align}
When $\mu=0$, \eqref{E-sol-ast-eta0} reduces to the $\ast$-RS equation.
In \cite{DSB-2024,ven-2022}, $\ast$-$\eta$-RS on KM $M^{2n+1}(f,\xi,\eta,g)$ were studied, and conditions when $g$ is 
an Einstein metric were found.

\smallskip

In  \cite{RWo-2,rov-119}, we defined metric structures on a smooth manifold
that generalize Hermitian, in parti\-cular K\"{a}hler, structure, as well as the almost contact metric structure and its important classes: Sasakian, cosymplectic, Kenmotsu, etc. structures.
These so-called ``weak" metric structures (the complex structure
is replaced by a nonsingular skew-symmetric tensor,
see Definition~\ref{D-basic}) allow a fresh look at the classical theory and find novel applications.
Weak contact metric manifolds form a broad class; for example, a twisted product $\mathbb{R}\times_\sigma\bar M$
of a real line and a weak K\"{a}hler manifold (Definition~\ref{D-wK}) is a weak $\beta$-KM (Definition~\ref{D-wK2}) with $\beta=(\partial_t\sigma)/\sigma$;
when $\beta=0$, and the metric product of $\mathbb{R}$ and a weak K\"{a}hler manifold is a weak cosymplectic~manifold.

In \cite{rst-55,rov-126}, we studied the interaction of Ricci-type solitons with weak (K-contact and $\beta$-Kenmotsu) metric structures, and proved that under certain conditions we get $\eta$-Einstein or even Einstein~metrics.

In this paper, we adapt the concept of the $\ast$-Ricci tensor to the weak almost contact metric structure and study
the question: when does a weak $\beta$-KM equipped with a $\ast$-$\eta$-RS (in particu\-lar, $\ast$-RS) carry an Einstein metric.
Sections~\ref{sec:01} and \ref{sec:02} contain preliminary material on weak almost contact metric structure and its distinguished class of weak $\beta$-KM.
Theorem~\ref{T-2.1} and Example~\ref{Ex-warped} show (using weak Hermitian and weak K\"{a}hler structures) that there are many purely weak $\beta$-KM.
Section~\ref{sec:03} introduces the $\ast$-Ricci tensor and $\ast$-scalar curvature and establishes their relation with the classical tensors (see Theorem~\ref{T-1-ast}).
We~also introduce  $\ast$-$\eta$-RS and give Example~\ref{Ex-ast-RS}.
Section~\ref{sec:04} contains our results on the interaction of $\ast$-$\eta$-RS, in particular, $\ast$-RS, with the $\beta$-KM when $\beta$ is constant.
Proposition~\ref{T-lambda0} derives the sum of the soliton constants, $\lambda+\mu$, under certain conditions.
Theorems~\ref{T-002a}--\ref{T-005} generalize the results of other authors and show that for a weak $\beta$-KM equipped with a $\ast$-$\eta$-RS, under certain conditions (the potential vector field $V$ is a contact vector field; $g$ is an $\eta$-Einstein metric; $V$ is collinear with $\xi$; and the soliton is a gradient almost $\ast$-$\eta$-soliton) $g$ is an Einstein metric.

\section{Preliminaries}
\label{sec:01}

In this section, we recall necessary definitions and results, see \cite{rst-55,RWo-2,rov-126}.

\begin{definition}\label{D-basic}
\rm
A~\textit{weak almost contact metric structure} on a smooth manifold $M^{2n+1}$ is a set $(f,Q,\xi,\eta,g)$ where
${f}$ is a $(1,1)$-tensor of rank $2\,n$, $Q$ is a nonsingular $(1,1)$-tensor,
$\xi$ is characteristic vector field, ${\eta}$ is a 1-form, $g$ is a Riemannian metric on $M$, satisfying
\begin{align}\label{2.1}
 & {f}^2 = -Q + {\eta}\otimes {\xi},\qquad {\eta}({\xi})=1,\qquad Q\,{\xi} = {\xi}, \\
\label{2.2}
 & g({f} X,{f} Y)= g(X,Q\,Y) - {\eta}(X)\,{\eta}(Y)\quad (X,Y\in\mathfrak{X}_M),
\end{align}
and $M^{2n+1}({f},Q,{\xi},{\eta},g)$ is called a \textit{weak almost contact metric manifold}.
\end{definition}

We assume that a smooth $2n$-dimensional distribution ${\cal D}:=\ker\eta$ is ${f}$-invariant, i.e., ${f} X\in{\cal D}\ (X\in{\cal D})$,
as in the classical theory \cite{blair2010riemannian}, where $Q={\rm id}_{\,TM}$.
By the above, the distribution ${\cal D}$ is $Q$-invariant: $Q({\cal D})={\cal D}$.
Setting $Y=\xi$ in \eqref{2.2}, we obtain $\eta(X)=g(X,\xi)$.

\begin{proposition}
 For a weak almost contact metric structure $({f},\xi,\eta,Q,g)$, the tensor ${f}$ is skew-symmetric, $Q$ is self-adjoint and the following is true:
\[
 {f}\,\xi=0,\quad \eta\circ{f}=0,\quad [Q,\,{f}]=0,\quad \eta\circ Q=\eta.
\]
\end{proposition}

Recall the following formulas (with $X,Y\in\mathfrak{X}_M$):
\begin{align}\label{3.3B}
 (\pounds_{X}\,\eta)(Y) & = X(\eta(Y)) - \eta([X, Y]) ,\\
 \label{3.3C}
 (\pounds_{Z}\,g)(X,Y) & = g(\nabla_X Z, Y) + g(X, \nabla_Y Z).
\end{align}

A weak almost contact metric structure $({f},Q,\xi,\eta, g)$ on a manifold $M$ is said to be {\it normal} if
the following tensor is identically zero:
\begin{align*}
 N^{(1)}(X,Y) := [{f},{f}](X,Y) + 2\,d\eta(X,Y)\,\xi,\quad X,Y\in\mathfrak{X}_M,
\end{align*}
where the exterior derivative $d\eta$ of $\eta$ and
the Nijenhuis torsion $[{f},{f}]$ of ${f}$
are given~by
\begin{align*}
 & d\eta(X,Y) = (1/2)\,\{X(\eta(Y)) - Y(\eta(X)) - \eta([X,Y])\},\quad X,Y\in\mathfrak{X}_M, \\
 & [{f},{f}](X,Y) := {f}^2 [X,Y] + [{f} X, {f} Y] - {f}[{f} X,Y] - {f}[X,{f} Y],\quad X,Y\in\mathfrak{X}_M.
\end{align*}
The~following three tensors $N^{(2)}, N^{(3)}$ and $N^{(4)}$ are well known, see \cite{blair2010riemannian}:
\begin{align*}
 N^{(2)}(X,Y) &:= (\pounds_{{f} X}\,\eta)(Y) - (\pounds_{{f} Y}\,\eta)(X) = 2\,d\eta({f} X, Y) - 2\,d\eta({f} Y, X), \\
 N^{(3)}(X) &:= (\pounds_{\xi}\,{f})(X)
 = [\xi, {f} X] - {f} [\xi, X],\\
 N^{(4)}(X) &:= (\pounds_{\xi}\,\eta)(X)
 = \xi(\eta(X)) - \eta([\xi, X]) = 2\,d\eta(\xi, X).
\end{align*}

\begin{proposition}[see \cite{rov-119}]
The normality condition ${\cal N}^{\,(1)}=0$ for a weak almost contact metric structure implies
${\cal N}^{\,(3)} = {\cal N}^{\,(4)} = 0$, ${\cal N}^{\,(2)}(X,Y) = \eta([\widetilde QX, fY])$ and
 $[X,\xi]\in{\cal D}$ for all $X\in{\cal D}$.
Moreover, $\nabla_{\xi}\,\xi=0$, that is, ${\cal D}^\bot$ defines a geodesic foliation.
\end{proposition}

Define a (1,1)-tensor $\widetilde Q := Q - {\rm id}$
and note that $[\widetilde{Q},{f}]=0$.
We also obtain ${f}^3+{f} = -\widetilde{Q}\,{f}$.



\section{Geometry of weak $\beta$-KM}
\label{sec:02}

In this section, we recall some well-known results, see \cite{rst-55,rov-126}.

\begin{definition}\label{D-wK}\rm
A weak almost contact metric manifold $M^{2n+1}(f,Q,\xi,\eta,g)$ is called a \textit{weak $\beta$-KM}
(a \textit{weak KM} when $\beta\equiv1$)~if \eqref{2.3-patraA} is valid:
\begin{equation}\label{2.3-patra}
 (\nabla_{X} f)Y=\beta\{g(f X, Y)\,\xi -\eta(Y)\,f X\}\quad (X,Y\in\mathfrak{X}_M),
\end{equation}
where $\beta$ is a non-zero smooth function on $M$; and if $n=1$, then we also require $d\beta\wedge\eta=0$.
\end{definition}

Taking $X=\xi$ in \eqref{2.3-patra} and using ${f}\,\xi=0$, we get $\nabla_{\xi}\,f=0$.

\begin{lemma}
The following formulas are true for a weak $\beta$-KM $M^{2n+1}(f,Q,\xi,\eta,g)$ and $X,Y\in\mathfrak{X}_M$:
\begin{align}\label{2.3b}
 & \nabla_{X}\,\xi = \beta\{X -\eta(X)\,\xi\}, \\
\label{2.3c}
 & (\nabla_{X}\,\eta)(Y) = \beta\{g(X,Y) -\eta(X)\,\eta(Y)\} ,\\
 \label{E-nS-10b}
 & (\nabla_X\,Q)Y = -\beta\big\{\eta(Y)\,\widetilde Q X + g(\widetilde Q X, Y)\,\xi\big\} .
\end{align}
In particular, $\nabla_\xi\,Q=0$ and $\tr Q=const$.
\end{lemma}

\begin{proof}
For the proof of \eqref{2.3b} and \eqref{2.3c} see \cite{rov-126}.
Using \eqref{2.1}, we find
\begin{align}\label{E-nS-10b1}
 (\nabla_X\,Q)Y = \nabla_X\,(QY)-Q\nabla_X\,Y = -(\nabla_X\,f^2)Y + \eta(Y)\nabla_X\,\xi + (\nabla_X\,\eta)(Y)\,\xi
  \quad (X,Y\in\mathfrak{X}_M).
\end{align}
Then using \eqref{2.3-patra}, \eqref{2.3b} and the equality $Q={\rm id}+\widetilde Q$ in \eqref{E-nS-10b1}, we find
\begin{align*}
\notag
 (\nabla_X\,Q)Y &= -f(\nabla_X f)Y -(\nabla_X f)fY + \beta\eta(Y)\{X-\eta(X)\,\xi\} +\beta\{g(X,Y)-\eta(X)\,\eta(Y)\}\,\xi \\
 & = -\beta\,\eta(Y)\,\widetilde Q X -\beta\,g(\widetilde Q X, Y)\,\xi \quad (X,Y\in\mathfrak{X}_M),
\end{align*}
that completes the proof of \eqref{E-nS-10b}.
\end{proof}

Using \eqref{2.3b} for a weak $\beta$-KM we obtain $\nabla_\xi\,\xi=0$ and
\begin{equation}\label{2.3d}
 \pounds_\xi\,g = 2\,\beta\{g - \eta\otimes\eta\} .
\end{equation}


\begin{proposition}[see \cite{rov-126}]
\label{P-2.1}
A weak $\beta$-KM satisfies the following conditions:
\begin{align*}
 {\cal N}^{\,(1)} = 0,\quad
 d\eta = 0,\quad
 d\Phi = 2\,\beta\,\eta\wedge\Phi .
\end{align*}
where the {fundamental $2$-form} $\Phi$ is defined by
 $\Phi(X,Y)=g(X,{f} Y)$ for all $X,Y\in\mathfrak{X}_M$.
\end{proposition}

By Proposition~\ref{P-2.1}, we get
\[
 0= d^2\Phi=2\,d\beta\wedge\eta\wedge\Phi.
\]
Thus, if $\dim M>3$, then the condition $d\beta\wedge\eta=0$ follows from \eqref{2.3-patra}.


\begin{definition}\label{D-wK2}\rm
A \textit{weak Hermitian structure} on an even-dimensional Riemannian manifold
$(\bar M, \bar g)$, equipped with a non-degenerate skew-symmetric {\rm (1,1)}-tensor $J$ is defined by condition $J^{\,2}<0$.
If~$\bar\nabla J=0$, where $\bar\nabla$ is the Levi-Civita connection of $\bar g$, then such~$(\bar g, J)$ is called a \textit{weak K\"{a}hler structure} on $M$.
\end{definition}

L.\,P. Eisenhart \cite{E-1923} proved that if a Riemannian manifold $(\bar M, \bar g)$ admits a pa\-rallel symmetric 2-tensor other than the constant
multiple of $\bar g$, then it is reducible.
Some authors, e.g., \cite{H-2022}, studied and classified (skew-)symmetric parallel 2-tensors on a Riemannian manifold.

\begin{example}\rm
According to the results in \cite{E-1923}, a weak K\"{a}hler~manifold with $J^2\ne -c^2\,{\rm id}$, where $c\in\mathbb{R}$, is reducible.
Take two (or even more) Hermitian manifolds $(\bar M_i,\bar g_i, J_i)$, hence $J_i^2=-{\rm id}_{\,i}$.
The product $\prod_{\,i}(\bar M_i,\bar g_i, c_i J_i)$, where $c_i\not\in\{0,\pm1\}$ are different constants, is a weak
Hermitian manifold with $Q=\bigoplus_{\,i}c_i^2\,{\rm id}_{\,i}$.
Moreover, if $(\bar M_i,\bar g_i, J_i)$ are K\"{a}hler~manifolds, then $\prod_{\,i}(\bar M_i,\bar g_i, c_i J_i)$ is a weak K\"{a}hler~manifold.
\end{example}

\begin{proposition}[see \cite{rov-126}]\label{T-2.1}
Every point of a weak $\beta$-KM $M({f},Q, \xi,\eta,g)$ has a neighborhood $U$ that is isometric to a twisted product $(-\varepsilon,\varepsilon)\times_\sigma\bar M$
$($a warped product if $\nabla\beta$ is collinear with $\xi)$, where $(\bar M,\bar g, J)$ is a weak Hermitian manifold  (weak K\"{a}hler manifold, respectively) and $(\partial_t\,\sigma)/\sigma=\beta$ is true.
Moreover, if $M$ is simply connected and complete, then the isometry is global.
\end{proposition}

\begin{example}\label{Ex-warped}\rm
A {warped product} $\mathbb{R}\times_\sigma\bar M$
of a Riemannian manifold $(\bar M,\bar g)$ and a real line $\mathbb{R}$
is the product $M=\mathbb{R}\,\times\bar M$ with the metric $g=dt^2\oplus \sigma^2\,\bar g$, where $\sigma>0$ is a smooth function on $\mathbb{R}$.
(If $\sigma>0$ is a smooth function on $M$, then we get a twisted product).
Set $\xi=\partial_t$.
For a {warped product}, the~Levi-Civita connections, $\nabla$ of $g$ and $\bar\nabla$ of $\bar g$, are related as follows:

\noindent\ \
(i) $\nabla_\xi\,\xi= 0$, $\nabla_X\,\xi=\nabla_\xi\,X= \xi(\log\sigma)X$ for $X\perp\xi$.

\noindent\ \
(ii) $\pi_{1*}(\nabla_XY) = -g(X,Y)\,\xi(\log\sigma)$, where $\pi_1: M \to\mathbb{R}$ is the orthoprojector.

\noindent\ \
(iii) $\pi_{2*}(\nabla_XY)$ is the lift of $\bar\nabla_XY$, where $\pi_2: M \to\bar M$ is the orthoprojector.

\smallskip\noindent
Submanifolds $\{t\}\times\bar M$ (called the \textit{fibers}) are totally umbilical, i.e., the Weingarten operator $A_\xi=-\nabla\xi$ on ${\cal D}$
(the distribution on $M$ orthogonal to $\xi$) is conformal with the factor $-\xi(\log\sigma)$, see (ii).
Note that $\sigma$ is constant along the fibers; thus, $X(\sigma)=\xi(\sigma)\,\eta(X)$ for all $X\in\mathfrak{X}_M$.

Let $(\bar M,\bar g, J)$ be a weak Hermitian manifold,
(e.g., a 4-dimensio\-nal Lie group, whose metric Lie algebra belongs to \cite[Table~4]{H-2022})
and $\partial_t\sigma\ne0$.
Then the warped product $\mathbb{R}\times_\sigma\bar M$ is a weak $\beta$-KM $M(f,\xi,\eta,g)$
with $\beta = (\partial_t\sigma)/\sigma$ and the following tensors on $M=\mathbb{R}\times\bar M$:
\begin{equation*}
 {f} = \left(\hskip-1mm
             \begin{array}{cc}
               0 & 0 \\
               0 & J \\
             \end{array}
           \hskip-1mm\right),\ \
 Q = \left(\hskip-1mm
             \begin{array}{cc}
               0 & 0 \\
               0 & - J^{\,2}
\\
             \end{array}
           \hskip-1mm\right),\ \
           \xi = \left(\hskip-1mm
            \begin{array}{c}
               \partial_t  \\
               0 \\
             \end{array}
           \hskip-2mm\right),\ \
        \eta = (dt, 0) ,\ \
           g = \left(\hskip-1mm
             \begin{array}{cc}
              dt^2 & 0 \\
               0 &
\sigma^2\bar g \\
             \end{array}
           \hskip-1mm\right).
\end{equation*}
Note that $X(\beta)=0\ (X\perp\xi)$.
 If $X,Y\in{\cal D}$, then
 $\pi_{2*}((\nabla_X{f})Y) =(\bar\nabla_X J)Y=0$
 and
 $\pi_{1*}((\nabla_X{f})Y) =-\beta\,g(X,{f} Y)$.
If~$X=\xi$ and $Y\in{\cal D}$, then
$(\nabla_\xi{f})Y=\beta({f} Y)-{f}(\beta Y)=0$.
If~$X\in{\cal D}$ and $Y=\xi$, then
\[
 (\nabla_X{f})\xi=-{f}\nabla_X\,\xi=-\beta\,{f} X.
\]
Also, we get $(\nabla_\xi{f})\,\xi=0$. By the above, the condition \eqref{2.3-patra} is true.
\end{example}

Let $\Ric^\sharp$ be a~(1,1)-tensor adjoint to the {Ricci tensor}
-- the suitable trace of the curvature tensor:
\begin{equation*}
 {\rm Ric}\,(X,Y) = {\rm trace}_{\,g}(Z\to R_{\,Z,X}\,Y),\quad
 R_{X, Y} = [\nabla_{X},\nabla_{Y}] - \nabla_{[X,Y]}\quad (X,Y,Z\in\mathfrak{X}_M).
\end{equation*}
The scalar curvature of a Riemannian manifold $(M,g)$ is defined as ${r}=\tr_{\,g}\Ric=\tr\Ric^\sharp$.

In the rest of this paper we assume $\beta=const\ne0$.

\begin{lemma}[see Lemma 3 of \cite{rov-126}]
The following formulas hold on a weak $\beta$-KM:
\begin{align}\label{2.4}
 & R_{X, Y}\,\xi = \beta^2(\eta(X)Y - \eta(Y)X)\quad (X,Y\in\mathfrak{X}_M),\\
\label{2.5-patra}
 & \Ric^\sharp \xi = -2\,n\beta^2\xi ,\\
 \label{E-L-02b}
 & (\nabla_\xi\Ric^\sharp) X = -2\,\beta\Ric^\sharp X -4\,n\beta^3 X \quad (X\in\mathfrak{X}_M) ,\\
\label{3.1A-patra}
 &\xi({r}) = -2\,\beta\,({r}+2\,n(2\,n+1)\beta^2) .
\end{align}
Differentiating \eqref{2.5-patra} along $X$ and using \eqref{2.3b} gives 
\begin{align}\label{E-L-02a}
 (\nabla_X\Ric^\sharp)\xi=-\beta\Ric^\sharp X-2\,n\beta^3 X\quad (X\in\mathfrak{X}_M) .
\end{align}
\end{lemma}

\begin{lemma}[see  Lemma 5 of \cite{rov-126}]
On an $\eta$-Einstein weak $\beta$-KM, see \eqref{E-eta-Einstein},
the expression of the Ricci tensor is the following:
\begin{align}\label{E-42-proc}
\notag
 & {\rm Ric} = \big(\frac{{r}}{2\,n}+\beta^2\big)\,g -\big(\frac{{r}}{2\,n}+(2n+1)\beta^2\big)\eta\otimes\eta , \\
 & {\rm Ric}^\sharp X = \big(\frac{{r}}{2\,n}+\beta^2\big)X -\big(\frac{{r}}{2\,n}+(2n+1)\beta^2\big)\eta(X)\,\xi \quad (X\in\mathfrak{X}_M).
\end{align}
\end{lemma}

\section{The $\ast$-Ricci Tensor of Weak $\beta$-Kenmotsu Manifolds}
\label{sec:03}

Here, we define the $\ast$-Ricci tensor and $\ast$-$\eta$-{RS} for week almost contact metric mani\-folds,
establish relation of $\ast$-Ricci tensor with the classical tensors and give Example~\ref{Ex-ast-RS}.

We define $\ast$-{\em Ricci tensor} of a weak almost contact metric manifold $M^{2n+1}(f,Q,\xi,\eta,g)$ using the same formula, see \eqref{E-ast-Ricci-tensor}, as for almost contact metric manifolds:
\begin{align*}
 \Ric^\ast(X,Y)
 =(1/2)\,\tr\{Z\to f\,R_{X, fY} Z\}
 = (1/2)\,\tr\{Z\to R_{X, fY} fZ\}
  \quad (X,Y\in\mathfrak{X}_M).
\end{align*}
Note that $\Ric^\ast$ is not symmetric.
The $\ast$-scalar curvature is given by
\[
{r}^\ast=\tr_{g}\Ric^\ast=\sum\nolimits_{\,i}\Ric^\ast(e_i,e_i),
\]
where $e_i\ (1\le i\le 2n+1)$ is a local orthonormal basis of $TM$.
A weak almost contact metric manifold $M(f,Q,\xi,\eta,g)$
will be called $\ast$-$\eta$-{\em Einstein manifold} ($\ast$-{\em Einstein manifold} when $\mu=0$)~if
\begin{align}\label{E-ast-Einstein}
 \Ric^\ast = \lambda\,g + \mu\,\eta\otimes\eta
\end{align}
for some $\lambda,\mu\in C^\infty(M)$.
Thus, a weak almost contact metric manifold is $\ast$-Einstein  manifold if
$\Ric^\ast = \lambda\,g$ for some $\lambda\in C^\infty(M)$.
Note that ${r}^\ast$ is not necessarily constant on a $\ast$-Einstein manifold.

The following result generalizes Proposition~10 of \cite{kenmotsu1972class}.

\begin{proposition}
 On a weak $\beta$-KM, we have
\begin{align}\label{E-lem2-1}
 & R_{X,Y}fZ - f\,R_{X,Y}Z = \beta^2\{
 g(Y,Z)\,fX - g(X,Z)\,fY + g(X,fZ)\,Y - g(Y,fZ)\,X \} ,\\
 \label{E-lem2-2}
 & R_{fX,fY}Z - R_{X,QY}Z = \beta^2\{
 g(QZ,Y)\,X - g(X,Z)\,QY + g(Z,fX)\,fY - g(Z,fY)\,fX \}.
\end{align}
\end{proposition}

\begin{proof} Recall the Ricci identity (commutation formula),
for example \cite{CLN-2006}:
\begin{align*}
 \big(\nabla_X\nabla_Y f - \nabla_Y\nabla_X f
 -\nabla_{[X,Y]} f\big)Z =
 (R_{X,Y}f)Z = R_{X,Y} fZ - f R_{X,Y} Z.
\end{align*}
From this and \eqref{2.3-patra} we obtain \eqref{E-lem2-1}.
Using \eqref{E-lem2-1} and skew-symmetry of $f$, we derive
\begin{align*}
& g(R_{fZ, fW}X, Y) + g( R_{X,Y}Z, f^2W)
 = g(R_{X,Y}fZ, fW) - g(f R_{X,Y}Z, fW) \\
& = \beta^2\{ g(X,Z)\,g(Y, f^2W) - g(Y,Z)\,g(X, f^2W)
 + g(X,fZ)\,g(Y, fW) - g(Y,fZ)\,g(X, fW) \} \\
& = \beta^2\{ g(Y,Z)\,g(X, QW) -g(X,Z)\,g(Y, QW) +g(X,Z)\,\eta(W)\eta(Y) -g(Y,Z)\,\eta(W)\eta(X) \\
& + g(X,fZ)\,g(Y, fW) - g(Y,fZ)\,g(X, fW) \}.
\end{align*}
Then, using \eqref{2.4}, we find \eqref{E-lem2-2}.
\end{proof}

For a weak almost contact metric manifold $M(f,Q,\xi,\eta,g)$,
we will build an $f$-\textit{basis}, see \cite{rov-128},
consis\-ting of mutually orthogonal nonzero vectors at a point $x\in M$.
Let $e_1\in(\ker\eta)_x$ be a unit eigenvector of the self-adjoint operator $Q>0$ with the eigenvalue $\lambda_1>0$.
Then, $fe_1\in(\ker\eta)_x$ is orthogonal to $e_1$ and $Q(fe_1) = f(Qe_1) = \lambda_1 fe_1$.
Thus, the subspace orthogonal to the plane $span\{e_1,fe_1\}$ is $Q$-invariant.
 There exists a unit vector $e_2\in(\ker\eta)_x$ such that $e_2\perp span\{e_1, fe_1\}$
and $Q\,e_2= \lambda_2 e_2$  for some $\lambda_2>0$.
Obviously, $Q(fe_2) = f(Q\,e_2) = \lambda_2 fe_2$.
All five vectors $\{e_1, fe_1,e_2, fe_2, \xi\}$ are nonzero and mutually orthogonal.
Continuing in the same manner, we find a basis $\{e_1, fe_1,\ldots, e_n, fe_n, \xi\}$ of $T_x M$
consisting of mutually orthogonal vectors.
Note that $g(fe_i,fe_i)=g(Qe_i,e_i)=\lambda_i$.

\begin{theorem}\label{T-1-ast}
For a weak $\beta$-KM $M^{2n+1}(f,Q,\xi,\eta,g)$, the $\ast$-{Ricci tensor} and the $\ast$-{scalar curvature}
are related to the Ricci tensor and the scalar curvature by the formulas
\begin{align}\label{E-ast-Ric}
  & \Ric^\ast(X,Y) = \Ric(X,QY) + \beta^2\{(2n-1) g(X,QY) + \eta(X)\,\eta(Y)\},\\
\label{E-ast-scal}
  & {r}^\ast = \tr(Q\Ric^\sharp)
  + \beta^2\{4\,n^2 + (2n-1)\tr\widetilde Q\} .
\end{align}
\end{theorem}

\begin{proof}
Using the first Bianchi identity, we obtain
\begin{align}\label{E-th1-a}
 2\Ric^\ast(X,Y) = \tr\{Z\to -f\,R_{fY,Z}X\}
 + \tr\{Z\to -f\,R_{Z, X} fY\}.
\end{align}
Let  $e_i\ (1\le i\le 2n+1)$ be a local $f$-basis of $TM$ and $e_{2n+1}=\xi$. Then $e'_i=fe_i/\|fe_i\|\ (1\le i\le 2n)$ is a local basis of ${\cal D}$.
Using \eqref{2.1}, \eqref{2.4}, the skew-symmetry of $f$ and the symmetries of the
curvature tensor, we show that the two terms in the right-hand side of \eqref{E-th1-a} are equal:
\begin{align}\label{E-th1-b}
\notag
 & \tr\{Z\to -f\,R_{fY,Z}X\}
 = -\sum\nolimits_{\,i} g(f R_{fY, e'_i}X, e'_i)
 = \sum\nolimits_{\,i} g(R_{fY, fe_i/\|fe_i\|}X, f^2 e_i/\|fe_i\|) \\
 \notag
 & = \sum\nolimits_{\,i} g(R_{fY, fe_i/\|fe_i\|}X,
 -\lambda_i e_i/\|fe_i\|)
 - \sum\nolimits_{\,i} \eta(e_i)/\|fe_i\|
 g(R_{fY, fe_i/\|fe_i\|}\,\xi, X) \\
 & = -\sum\nolimits_{\,i} (\lambda_i/\|fe_i\|^2)g(R_{fY, fe_i}X, e_i)
 = \sum\nolimits_{\,i} g(R_{e_i,X}fY, fe_i)
 = \tr\{Z\to -f\,R_{Z, X} fY\}.
\end{align}
Using \eqref{E-lem2-2} with $Z=e_i$, we get
\begin{align}\label{E-th1-c}
\notag
& \tr\{Z\to -f\,R_{Z, X} fY\}
= -\sum\nolimits_{\,i} g(R_{e_i,X}f e_i, fY) \\
& =-\sum\nolimits_{\,i} g(R_{e_i,X} e_i, QY) \}
 -\beta^2\{ - 2\,n\,g(X,QY)+ g(fX, fY)\}.
\end{align}
From \eqref{E-th1-a}, \eqref{E-th1-b} and \eqref{E-th1-c} we obtain \eqref{E-ast-Ric}.
Contracting \eqref{E-ast-Ric} yields \eqref{E-ast-scal}.
\end{proof}

Using Theorem~\ref{T-1-ast}, we can express $\Ric^\ast$ in terms of ${r}^\ast$ on $\beta$-KM as follows.

\begin{corollary}
On a $\ast$-$\eta$-Einstein weak $\beta$-KM the expression of the $\ast$-Ricci tensor is the following:
\begin{align*}
 \Ric^\ast = \frac{{r}^\ast}{2\,n}\big\{g - \eta\otimes\eta\big\} .
\end{align*}
\end{corollary}

\begin{proof}
Tracing \eqref{E-ast-Einstein} gives ${r}^\ast=(2\,n+1)\lambda+\mu$.
Putting $X=Y=\xi$ in \eqref{E-ast-Einstein} and using \eqref{E-ast-Ric} yields $\lambda+\mu = 0$. From the above equalities, we express parameters of \eqref{E-ast-Einstein}
as follows: $\lambda=-\mu=\frac{{r}^\ast}{2\,n}$.
\end{proof}

\begin{definition}\rm
Given $\mu,\lambda\in\mathbb{R}$ and a vector field $V$ on a weak almost contact manifold $M(f,Q,\xi,\eta,g)$,
a $\ast$-$\eta$-{\em RS} is defined by the same formula as \eqref{E-sol-ast-eta0},
\begin{align}\label{E-sol-ast-eta}
 (1/2)\,\pounds_V\,g+\Ric^\ast = \lambda\,g + \mu\,\eta\otimes\eta .
\end{align}
If $\mu=0$ in \eqref{E-sol-ast-eta}, then we get a $\ast$-{\em RS} equation on a weak almost contact manifold,
\begin{align}\label{E-sol-ast}
 (1/2)\,\pounds_V\,g+\Ric^\ast =\lambda\,g.
\end{align}
A $\ast$-$\eta$-RS \eqref{E-sol-ast-eta} or a $\ast$-{RS} \eqref{E-sol-ast}
is called \textit{expanding}, \textit{steady}, or \textit{shrinking} if $\,\lambda$, is negative, zero, or positive, respectively.
\end{definition}

Note that the $\ast$-$\eta$-Ricci tensor $\Ric^\ast$ in \eqref{E-sol-ast-eta} or \eqref{E-sol-ast} is necessarily symmetric.
If $V$ is a Killing vector field, i.e., $\pounds_V\,g=0$, then the solitons of \eqref{E-sol-ast-eta} or \eqref{E-sol-ast} are trivial and $g$ becomes an $\ast$-$\eta$-{\em Einstein} metric or an $\ast$-{\em Einstein} metric, respectively.

For a weak $\beta$-KM, using the formula of $\ast$-Ricci tensor \eqref{E-ast-Ric} in the $\ast$-$\eta$-RS equation \eqref{E-sol-ast-eta}, we~get
\begin{align}\label{E-sol-eta2}
 (1/2)\,(\pounds_V g)(X,Y) + \Ric(X,QY) = \lambda\,g(X,Y) - (2n-1)\beta^2 g(X,QY) + (\mu-\beta^2)\eta(X)\eta(Y) ,
\end{align}
where $\mu,\lambda\in\mathbb{R}$; in this case, $Q$ commutes with the Ricci tensor $\Ric^\sharp$: $[\Ric^\sharp, Q]=0$.

The following example of a weak $\beta$-KM and $\ast$-$\eta$-{RS} will illustrate our main results in Section~\ref{sec:04}.

\begin{example}\label{Ex-ast-RS}\rm
Given real $\beta\ne0$, we consider linearly independent vector fields
\[
 e_i = e^{-\beta\,x_{2n+1}}\partial_i,\quad e_{2n+1} = \partial_{2n+1}\quad (1\le i\le 2n)
\]
in $M=\mathbb{R}^{2n+1}(x_1,\ldots,x_{2n+1})$, and define a Riemannian metric $g$ as $g(e_i,e_j)=\delta_{ij}$.
Set $\xi=e_{2n+1}$ and $\eta=dx_{2n+1}$. For $c=const>-1$ we define the (1,1)-tensors $f$ and $Q$ by
\begin{align*}
& f e_1 = \sqrt{1+c}\,e_{n+1},\ \
  f e_{n+1} = -\sqrt{1+c}\,e_1,\
   \ldots\
  f e_n =\sqrt{1+c}\,e_{2n},\ \
  f e_{2n} = -\sqrt{1+c}\,e_n,\ \
  f \xi =    0, \\
& Q e_1 = (1+c)\,e_1,\
  \ldots, \
  Q e_{2n} = (1+c)\,e_{2n},\quad
  Q\,\xi = \xi.
\end{align*}
Since \eqref{2.1}--\eqref{2.2} are valid, $M(f,Q,\xi,\eta,g)$ is a weak almost contact metric manifold when $c\ne0$ (an almost contact metric manifold when $c=0$). We can deduce that
\[
 [e_i,e_j]=0,\quad [e_i, \xi] = \beta e_i,\quad (1\le i,j\le 2n).
\]
Thus, ${\cal D}={\rm Span}(e_1,\ldots,e_{2n})$ is an involutive distribution.
The Levi-Civita connection $\nabla$ of $g$ on the vector fields $\{e_i\}$ is given by
$2\,g(\nabla_{e_i}e_j, e_k)= g(e_k,[e_i,e_j])-g(e_i,[e_j,e_k])-g(e_j,[e_i,e_k])$, and we~get
\begin{align*}
 \nabla_{e_i}e_j = -\delta_{ij}\,\beta\,\xi,\quad
 \nabla_{e_j}\,\xi = \beta e_j,\quad
 \nabla_{\xi}\,\xi = 0\quad (1\le i\le 2n+1,\ 1\le j\le 2n).
\end{align*}
Since \eqref{2.3-patra} is satisfied, $M(f,Q,\xi,\eta,g)$ is a weak $\beta$-KM when $c\ne0$ ($\beta$-KM when $c=0$).
Thus, the distribution ${\cal D}$ defines a totally umbilical foliation and $e_{2n+1}$ defines a geodesic foliation.
The second fundamental form of ${\cal D}$ is $h(e_i,e_j)=-\beta\delta_{ij}\,\xi\ (1\le i,j\le 2\,n)$
and the mean curvature vector of ${\cal D}$ is $H=-\beta\,\xi$.
The~non-vanishing components of the curvature tensor are
\begin{align*}
 & R_{\,e_i,e_j}e_j = - R_{\,e_j,e_i}e_j = - \beta^2 e_i,\quad
 R_{\,e_i,\,\xi}\,\xi = -\beta^2 e_i,\quad R_{\,\xi,\,e_j}e_j = -\beta^2\,\xi\quad (1\le i,j\le 2n).
\end{align*}
From the above results we have the Ricci tensor
\begin{align*}
 \Ric^\sharp e_i = -2\,n\beta^2 e_i,\quad \Ric^\sharp \xi = -2\,n\beta^2 \xi\quad (1\le i\le 2n).
\end{align*}
Thus, $g$ is an Einstein metric, $\Ric = -2\,n\beta^2 g$, of the scalar curvature $r=-2\,n(2\,n+1)\beta^2$.
Next we~find
\[
 \Ric^\ast(e_i,e_j)=-(1+c)\delta_{ij}\beta^2,\quad
 \Ric^\ast(\xi,\,\xi)=0,\quad
 \Ric^\ast(e_i,\,\xi)=0\quad (1\le i,j\le 2n).
\]
Hence $g$ is a $\ast$-$\eta$-Einstein metric: $\Ric^\ast=-(1+c)\beta^2\{g-\eta\otimes\eta\}$ of $\ast$-scalar curvature $r^*=-2\,n(1+c)\beta^2$.

We~can justify that $\pounds_\xi\,g=2\,\beta\{g-\eta\otimes\eta\}$ is valid.
Thus \eqref{E-sol-ast-eta} is valid for $V=\xi$ and $\lambda=-\mu=\beta-(1+c)\beta^2$.
We conclude that $(g,\xi)$ represents a $\ast$-$\eta$-RS on $M$.
\end{example}

\section{Main Results}
\label{sec:04}

In this section, we study the interaction of the $\ast$-$\eta$-RS with the weak $\beta$-KM.
In~Proposition~\ref{T-lambda0}, using auxiliary Lemma~\ref{lem3.2}, we~find the sum of the soliton constants,
and in Theorems 2-5 we present new characteristics of Einstein metrics.


Adapting the method of Lemma 4 in \cite{rov-126}, see also Lemma 3.2 in \cite{ven-2022}, we obtain the following.

\begin{lemma}\label{lem3.2}
Let a weak $\beta$-KM $M^{2n+1}(\varphi,Q,\xi,\eta,g)$ be a $\ast$-$\eta$-RS with the potential vector field $V$, then
\begin{align}\label{3.9}
\notag
 & (\pounds_{V} R)_{X,Y}\,\xi = 2\,\beta\{(\nabla_X\Ric^\sharp)Q Y - (\nabla_Y\Ric^\sharp)Q X\} \\
\notag
 & +2\,\beta^2 \{\eta(X)\Ric^\sharp Y-\eta(Y)\Ric^\sharp X\}
 + 4\,\beta^2\{\eta(X)\Ric^\sharp \widetilde QY - \eta(Y)\Ric^\sharp \widetilde QX\} \\
 & + 4\,n\,\beta^4\{\eta(X) Y - \eta(Y) X\} + 8\,n\beta^4\{\eta(X)\widetilde Q Y - \eta(Y)\widetilde Q X\}
\end{align}
for all $X,Y\in\mathfrak{X}_M$.
Using $Y=\xi$ and \eqref{3.1A-patra} in \eqref{3.9} gives
\begin{align}\label{E-3.18}
 (\pounds_V R)_{X,\,\xi}\,\xi = 0 \quad (X\in\mathfrak{X}_M).
\end{align}
\end{lemma}

\begin{proof}
Taking the covariant derivative of \eqref{E-sol-eta2} along $Z\in\mathfrak{X}_M$ and using \eqref{2.3-patra}, we~get
\begin{align}\label{3.3A-patra}
\nonumber
 (\nabla_Z(\pounds_{V}\,g))(X,Y) &= -2(\nabla_{Z}\,{\rm Ric})(X,QY)
 -2\Ric(X,(\nabla_Z\,Q)Y)-2(2n-1)\beta^2 g(X,(\nabla_Z\,Q)Y)\\
 & +\,2\,\beta(\mu-\beta^2)\{g(X, Z)\,\eta(Y) + g(Y,Z)\,\eta(X) -2\,\eta(X)\,\eta(Y)\,\eta(Z)\}
\end{align}
for all $X,Y\in\mathfrak{X}_M$.
Recall the commutation formula with the tensor $\pounds_{V}\nabla$, see \cite[p.~23]{yano1970integral},
\begin{align}\label{2.6}
 (\pounds_{V}(\nabla_{Z}\,g) - \nabla_{Z}(\pounds_{V}\,g) - \nabla_{[V,Z]}\,g)(X,Y) = -g((\pounds_{V}\nabla)(Z,X),Y) -g((\pounds_{V}\nabla)(Z,Y),X).
\end{align}
Since Riemannian metric is parallel, $\nabla g=0$, it follows from \eqref{2.6} that
\begin{equation}\label{3.4}
 (\nabla_{Z}(\pounds_{V}\,g))(X,Y) = g((\pounds_{V}\nabla)(Z,X),Y) + g((\pounds_{V}\nabla)(Z,Y),X).
\end{equation}
Since $(\pounds_{V}\nabla)(X,Y)$ is symmetric, from \eqref{3.4} we get
\begin{align}\label{3.5cycle}
 2\,g((\pounds_{V}\nabla)(X,Y), Z) =
 (\nabla_X\,\pounds_{V}\,g)(Y,Z)
 +(\nabla_Y\,\pounds_{V}\,g)(Z,X)
 -(\nabla_Z\,\pounds_{V}\,g)(X,Y) .
\end{align}
According to \eqref{3.5cycle} and \eqref{3.3A-patra}, we obtain
\begin{align}\label{3.6}
\nonumber
 g((\pounds_{V}\nabla)(X,Y),Z) &= (\nabla_{Z}\,{\rm Ric})(X,QY) -(\nabla_{X}\,{\rm Ric})(Y,QZ) -(\nabla_{Y}\,{\rm Ric})(Z,QX)\\
 & +2\,\beta(\mu-\beta^2)\,\eta(Z)\{g(X, Y) -\eta(X)\,\eta(Y) \} +\varepsilon(X,Y,Z),
\end{align}
where
\begin{align*}
 & \varepsilon(X,Y,Z) = \Ric(X,(\nabla_Z\,Q)Y) -\Ric(Z,(\nabla_Y\,Q)X) -\Ric(Y,(\nabla_X\,Q)Z) \\
 & +(2n-1)\beta^2 g(X,(\nabla_Z\,Q)Y) -(2n-1)\beta^2 g(Z,(\nabla_Y\,Q)X) -(2n-1)\beta^2 g(Y,(\nabla_X\,Q)Z).
\end{align*}
Substituting $Y=\xi$ in \eqref{3.6} yields:
\begin{align}\label{E-Lem3-a}
 g((\pounds_{V}\nabla)(X,\xi),Z) =(\nabla_{Z}\,{\rm Ric})(X,\xi)
 -(\nabla_{X}\,{\rm Ric})(\xi,QZ)
 -(\nabla_{\xi}\,{\rm Ric})(Z,QX) + \varepsilon(X,\xi,Z) ,
\end{align}
where, in view of \eqref{E-nS-10b}, we have
\begin{align*}
 \varepsilon(X,\xi,Z) = - \beta\Ric(\widetilde Q X, Z) - 2\,n\beta^3 g(\widetilde Q X, Z).
\end{align*}
Applying \eqref{E-L-02a} and \eqref{E-L-02b} to \eqref{E-Lem3-a}, we obtain
\begin{align}\label{3.7}
 (\pounds_{V} \nabla)(X,\xi) = 2\,\beta\Ric^\sharp Q X + 4\,n\beta^3 Q X
 \quad (X\in\mathfrak{X}_M).
\end{align}
Next, using (\ref{2.3b}) in the covariant derivative of (\ref{3.7}) for $Y\in\mathfrak{X}_M$, yields
\begin{align*}
 &(\nabla_Y(\pounds_V\nabla))(X,\xi) =
 - \beta(\pounds_V\nabla)(X,Y)
 + 2\,\beta(\nabla_Y\Ric^\sharp)Q X
 + 2\,\beta^2\eta(Y)\Ric^\sharp QX \\
 & -2\,\beta^2\eta(X)\Ric^\sharp\widetilde Q Y
 + 4\,n\beta^4\eta(Y)QX
 -4\,n\beta^4\eta(X)\widetilde Q Y
\end{align*}
for any $X\in\mathfrak{X}_M$.
Plugging this in the following formula with $Z=\xi$ (see \cite{yano1970integral}, p.~23):
\begin{align*}
 (\pounds_{V} R)_{X,Y} Z = (\nabla_{X}(\pounds_{V}\nabla))(Y,Z) -(\nabla_{Y}(\pounds_{V}\nabla))(X,Z),
\end{align*}
and using symmetry of $(\pounds_V\nabla)(X,Y)$, we deduce \eqref{3.9}.
Substituting $Y=\xi$ in \eqref{3.9} and using
\eqref{2.5-patra}, \eqref{E-L-02b} and \eqref{E-L-02a}, gives \eqref{E-3.18}.
\end{proof}

The following result generalizes Theorem~3.1 of \cite{ven-2019}  (see also Lemma 3.7 of \cite{ven-2022}).

\begin{proposition}\label{T-lambda0}
Let a weak $\beta$-KM be a $\ast$-$\eta$-RS.
Then $(\pounds_V\,\eta)(\xi) = \eta(\pounds_V\,\xi) = 0$ is true
and the soliton constants satisfy
\begin{align}\label{E-3.21b}
 \lambda +\mu = 0.
\end{align}
The $\ast$-$\eta$-RS is shrinking, steady, or expanding
if $\mu$ is negative, zero, or positive, respectively.
\end{proposition}

\begin{proof}
Using \eqref{2.4} and symmetries of $R$, we derive
\[
 g(R_{X,\xi}Z, W)=g(R_{W,Z}\xi, X)=\beta^2\{g(X,Z)g(\xi,W)-\eta(Z)g(X,W)\}.
\]
Hence
\begin{align}\label{E-3.19a}
 R_{X,\xi}Z = \beta^2\{g(X,Z)\xi - \eta(Z) X\}.
\end{align}
Taking the Lie derivative along $V$ of $R_{X,\xi}\,\xi= \beta^2\{\eta(X)\xi - X\}$, see \eqref{2.4} with $Y=\xi$,
and using \eqref{E-3.19a}, \eqref{3.3B} and \eqref{2.4} gives
\begin{align*}
 (\pounds_V R)_{X,\xi}\,\xi = 2\,\beta^2\eta(\pounds_V\,\xi)X - \beta^2 g(X,\pounds_V\,\xi)\,\xi + \beta^2(\pounds_V\,\eta)(X)\,\xi ,
\end{align*}
where $\pounds_V\,\eta$ is given in \eqref{3.3B}.
In view of \eqref{E-3.18}, the above equation divided by $\beta^2$ becomes
\begin{align}\label{E-3.19}
 2\,\eta(\pounds_V\,\xi)X - g(X,\pounds_V\,\xi)\,\xi + (\pounds_V\,\eta)(X)\,\xi = 0 .
\end{align}
Taking ${\cal D}$- and ${\cal D}^\bot$- components of \eqref{E-3.19} gives two equalities
\begin{align}\label{E-3.19X}
\notag
 \eta(\pounds_V\,\xi) & = 0 ,\\
 (\pounds_V\,\eta)(X) & = g(X,\pounds_V\,\xi) \quad (X\in\mathfrak{X}_M).
\end{align}
Using $X=\xi$ in the second equality of \eqref{E-3.19X} yields $(\pounds_V\,\eta)(\xi)=\eta(\pounds_V\,\xi)$.
By the above, we conclude that
\begin{align}\label{E-AB}
 (\pounds_V\,\eta)(\xi) = \eta(\pounds_V\,\xi) = 0.
\end{align}
Further,~using \eqref{2.5-patra}, we write \eqref{E-sol-eta2} with $Y=\xi$ as
\begin{align}\label{E-3.20}
 (\pounds_V\,g)(X,\xi) = 2\,\big(\lambda+\mu \big)\,\eta(X).
\end{align}
Using the equality
\[
 (\pounds_V\,g)(\xi,\xi)=-2\,g(\xi, \pounds_V\,\xi)=-2\,\eta(\pounds_V\,\xi)
\]
\eqref{E-AB} and $(\pounds_V\,g)(\xi,\xi) = 2\,\big(\lambda+\mu \big)$, see \eqref{E-3.20} with $X=\xi$, yields \eqref{E-3.21b}.
\end{proof}

\begin{corollary}
Let a weak $\beta$-KM be a $\ast$-RS, then the soliton constant is $\lambda = 0$.
\end{corollary}

The following result generalizes Theorem~3.2 of \cite{ven-2019} (see also Theorem~3.3 of \cite{ven-2022}).

\begin{theorem}
Let $M^{2n+1}(\varphi,Q,\xi,\eta,g)$ be an $\eta$-Einstein weak $\beta$-KM.
If~$(g,V)$ represents a $\ast$-$\eta$-RS, then $(M,g)$ is an Einstein manifold of scalar curvature ${r}=-2\,n(2\,n+1)\beta^2$.
\end{theorem}

\begin{proof}
Taking the covariant derivative of \eqref{E-42-proc} in the $Y$-direction and using \eqref{2.3b} and \eqref{2.3c} yields
\begin{align}\label{E-Th2-a}
 (\nabla_Y\Ric^\sharp)X = \frac{Y({r})}{2\,n}\{X-\eta(X)\xi\}
 -\big\{\frac{r}{2\,n}+(2\,n+1)\beta^2\big\}\beta\{g(X,Y)\xi -2\,\eta(X)\eta(Y)\xi +\eta(X)Y\}.
\end{align}
Contracting \eqref{E-Th2-a} over $Y$, we get, see also \cite{rov-126},
\begin{align}\label{E-Th2-d}
 (n-1)X({r}) = -\{\xi({r}) + 2\,n\beta({r}+2\,n(2\,n+1)\beta^2)\}\eta(X) .
\end{align}
Using \eqref{3.1A-patra} in \eqref{E-Th2-d}, we conclude that
\begin{align}\label{E-Th2-e}
 X({r}) = \xi({r})\eta(X) ;
\end{align}
hence, ${r}$ is constant along the leaves of ${\cal D}$.
Applying \eqref{E-42-proc}, \eqref{E-Th2-a} and \eqref{E-Th2-e} in \eqref{3.9}, gives
\begin{align}\label{E-Th2-b}
 (\pounds_V R)_{X,Y}\,\xi = 0 .
\end{align}
Contracting \eqref{E-Th2-b} over $X$ and using the assumptions, we get, see also \cite{rov-126},
\begin{align}\label{E-Th2-f}
 (\pounds_V \Ric)(Y,\xi) = \tr\{X \to (\pounds_V R)_{X,Y}\,\xi\} = 0.
\end{align}
Taking the Lie derivative of $\Ric(Y,\xi) = -2\,n\beta^2\eta(Y)$, see \eqref{2.5-patra}, along $V$ yields
\begin{align}\label{E-Th2-g}
 (\pounds_V \Ric)(Y,\xi) + \Ric(Y,\pounds_V\xi) = -2\,n\,\beta^2(\pounds_V\,\eta)(Y) .
\end{align}
Using \eqref{E-Th2-f} in the preceding equation \eqref{E-Th2-g}, we obtain
\begin{align}\label{E-Th2-fg}
 \Ric(Y,\pounds_V\xi) = -2\,n\,\beta^2(\pounds_V\,\eta)(Y) .
\end{align}
From \eqref{E-Th2-fg}, using \eqref{E-42-proc}, we find
\begin{align}\label{E-Th2-h}
 \big\{\frac{r}{2\,n} + \beta^2\big\} g(Y,\pounds_V\,\xi) - \big\{\frac{r}{2\,n} + (2\,n+1)\beta^2\big\}\eta(Y)\eta(\pounds_V\,\xi)
 =
 - 2\,n\beta^2(\pounds_V\,\eta)(Y) .
\end{align}
From \eqref{E-Th2-h}, using \eqref{E-3.19X}, we obtain the following equality (with two factors):
\begin{align}\label{E-Th2-i}
 \big\{{r} + 2\,n(2\,n+1)\beta^2\big\}\pounds_V\,\eta = 0 .
\end{align}

\textbf{Case 1}. Let
${r}=-2\,n(2\,n+1)\beta^2$ is valid on $M$, then by \eqref{E-42-proc}, $g$ is an Einstein metric.

\textbf{Case 2}. Let ${r}\ne-2\,n(2\,n+1)\beta^2$ on an open set ${\cal U}\subset M$.
Then \eqref{E-Th2-i} yields $\pounds_V\,\eta = 0$, i.e. $V$~is a strictly contact vector field on ${\cal U}$.
Let us show that this leads to a contradiction.
By $\pounds_V\,\eta = 0$ and \eqref{E-3.19X}, we get $\pounds_V\,\xi=0$.
Thus, from \eqref{E-ict-3} we find $(\pounds_V\nabla)(X,\xi)=0$.
Applying the previous equality to \eqref{3.7} yields $ \Ric^\sharp QX = - 2\,n\beta^2 QX$.
Since ${\cal D}$ is invariant for the non-singular operator $Q$, we get
\begin{align*}
 \Ric^\sharp X = - 2\,n\beta^2 X .
\end{align*}
Hence, ${r}=-2\,n(2\,n+1)\beta^2$ on ${\cal U}$ -- a contradiction.
\end{proof}

Recall the well-known formula, see \cite{ven-2022},
\begin{align}\label{E-ict-2}
 (\pounds_V\nabla)(X,Y) = \pounds_V\nabla_X\,Y - \nabla_X\pounds_V\,Y -\nabla_{[V,X]}\,Y  .
\end{align}

\begin{definition}\rm
A vector field $X$ on a weak almost contact manifold $M^{2n+1}({f},Q,\xi,\eta,g)$ is said to be a~\textit{contact vector field} if there is a smooth function $\sigma$ on $M$ such that
\begin{align}\label{E-ict-1}
 \pounds_X\,\eta = \sigma\,\eta .
\end{align}
Further, $X$ is called a \textit{strictly contact vector field} when $\sigma = 0$.
\end{definition}

We consider a weak KM as a $\ast$-$\eta$-RS (e.g., $\ast$-RS) with contact potential vector field $V$.

The following result generalizes Theorem~3.8 of \cite{ven-2022} (see also Theorem~3.3 and Corollary~3.4 of~\cite{DSB-2024}).

\begin{theorem}\label{T-002a}
Let a weak $\beta$-KM be a $\ast$-$\eta$-RS, whose potential vector field $V$ is a contact vector field.
Then $V$ is strictly contact and $(M, g)$ is an Einstein manifold of scalar curvature ${r}=-2\,n(2\,n{+}1)\beta^2$.
\end{theorem}

\begin{proof}
Taking the Lie derivative of $\eta(X)=g(X,\xi)$ along $V$ yields
\begin{align}\label{E-3.21L}
 (\pounds_V\,\eta)(X) = (\pounds_V\,g)(X,\xi) + g(X, \pounds_V\,\xi) .
\end{align}
Then, using \eqref{E-3.20} and \eqref{E-ict-1} in \eqref{E-3.21L}, we obtain
\[
 \pounds_V\,\xi = \big(\sigma- 2(\lambda+\mu)\big)\,\xi.
\]
Using this, \eqref{E-AB} and \eqref{E-3.21b}, we get $\sigma=\lambda+\mu=0$, hence $V$ is strictly contact.
Moreover,
\begin{align*}
 \pounds_V\,\xi = 0.
\end{align*}
Also, \eqref{E-ict-1} yields  $\pounds_V\,\eta=(\lambda+\mu)\,\eta=0$.
Setting $Y=\xi$ in \eqref{E-ict-2} and using \eqref{2.3b}, we find
\begin{align}\label{E-ict-3}
 (\pounds_V\nabla)(X,\xi) = -\beta\{(\pounds_V\,\eta)(X)\xi + \eta(X)\pounds_V\,\xi\} = 0.
\end{align}
Then, applying \eqref{3.7} yields
\begin{align*}
 \Ric^\sharp QX = - 2\,n\beta^2 QX .
\end{align*}
Since ${\cal D}$ is invariant for the non-singular operator $Q$, we get $\Ric^\sharp X = - 2\,n\beta^2 X$ for all $X\in{\cal D}$.
Hence $(M, g)$ is an Einstein manifold of scalar curvature ${r}=-2\,n(2\,n+1)\beta^2$.
\end{proof}

The following result generalizes Theorem 4.3 of \cite{ven-2022}.

\begin{theorem}\label{T-004}
Let a weak $\beta$-KM $M({f},Q,\xi,\eta,g)$ be a $\ast$-$\eta$-RS, whose potential vector field $V$ is collinear with $\xi$:
$V=\delta\,\xi$ for a smooth function $\delta\ne0$ on $M$. Then $\delta=const$ and $(M, g)$ is an
Einstein manifold of scalar curvature ${r}=-2\,n(2\,n+1)\beta^2$.
\end{theorem}

\begin{proof}
From \eqref{E-sol-eta2}, using \eqref{3.3C} and Proposition~\ref{T-lambda0}, we obtain
\begin{align}\label{E-T4-1}
\notag
 & g(\nabla_X V, Y) + g(X, \nabla_Y V) = - 2\Ric(X,QY) \\
 & - 2\,\mu g(X,Y) - 2(2n-1)\beta^2 g(X,QY) + 2(\mu-\beta^2)\eta(X)\,\eta(Y) .
\end{align}
Using $V=\delta\,\xi$ in \eqref{E-T4-1}, we obtain
\begin{align}\label{E-T4-2}
\notag
 & (d\delta)(X)\,\eta(Y) + (d\delta)(Y)\,\eta(X)
 + 2\,\delta\,\beta\{g(X,Y) - \eta(X)\,\eta(Y)\}
 = - 2\Ric(X,QY) \\
 & - 2\,\mu g(X,Y) - 2(2n-1)\beta^2 g(X,QY) + 2(\mu-\beta^2)\eta(X)\,\eta(Y) .
\end{align}
Substituting $Y=\xi$ in \eqref{E-T4-2} and using \eqref{2.3b} and \eqref{2.5-patra}, gives
\begin{align}\label{E-T4-3}
 (d\delta)(X) + (d\delta)(\xi)\,\eta(X) = 0 .
\end{align}
Therefore, $\delta$ is invariant along the distribution
$\ker\eta$, that is $(d\delta)(X)=0\ (X\perp\xi)$.
Changing $X$ to $\xi$ in \eqref{E-T4-3} yields $d\delta(\xi)=0$.
By the above, $\delta=const$. For $X,Y\perp\xi$ we obtain
\begin{align}\label{E-T4-4}
\notag
 & \delta \pounds_\xi\,g(X,Y) = - 2\Ric(X,QY) \\
 & - 2\,\mu g(X,Y) - 2(2n-1)\beta^2 g(X,QY) + 2(\mu-\beta^2)\eta(X)\,\eta(Y) .
\end{align}
Taking \eqref{2.3d} in \eqref{E-T4-4} and using \eqref{2.3d} gives
\begin{align*}
\notag
 & \Ric(X,QY) = -\delta\,\beta\{g(X,Y)-\eta(X)\,\eta(Y)\} \\
 & - \mu g(X,Y) - (2n-1)\beta^2 g(X,QY) + (\mu-\beta^2)\eta(X)\,\eta(Y) ,
\end{align*}
or, using \eqref{E-ast-Ric},
\begin{align*}
 \Ric^\ast = -(\mu+\delta\,\beta) g + (\mu + \delta\beta)\,\eta\otimes\eta,
\end{align*}
which means that our manifold is $\ast$-$\eta$-Einstein.
From the above equation with $\Ric(X,QY)$, we obtain
\begin{align}\label{E-T4-4b}
 \Ric(X,QY) = -(\delta\,\beta+\mu) g(X,Y) -(2\,n-1)\beta^2 g(X,QY) +(\delta\,\beta+\mu-\beta^2)\eta(X)\,\eta(Y) .
\end{align}
Taking the covariant derivative of \eqref{E-T4-4b} in $Z$ and using \eqref{E-nS-10b} yields
\begin{align}\label{E-T4-4c}
\notag
 & (\nabla_Z\Ric)(X,QY) = \beta\,\eta(Y)\big\{\Ric(X,\widetilde Q Z) + (2n-1)\beta^2 g(X,\widetilde Q Z)\big\}
 - \beta^3\eta(X) g(Y,\widetilde Q Z)\\
 & + (\delta\,\beta+\mu-\beta^2)\{ g(Z,X)\eta(Y) + g(Z,Y)\eta(X) - 2\,\eta(X)\eta(Y)\eta(Z)\} .
\end{align}
Using $Z=\xi$ in \eqref{E-T4-4c}, we obtain $(\nabla_\xi\Ric)(X,QY)=0$.
Applying to this \eqref{E-L-02b} and taking into account \eqref{2.5-patra} and non-singularity of $Q$, we obtain $\Ric^\sharp X=-2\,n\beta^2 X$.
Therefore, $g$ is an Einstein metric.
\end{proof}

If $v,\lambda,\mu\in C^\infty(M)$ and $V=\nabla v$, then \eqref{E-sol-ast-eta} gives a \textit{gradient almost $\ast$-$\eta$-RS}:
\begin{align}\label{E-sol-ast-eta-grad}
 {\rm Hess}_{\,v} +\Ric^\ast=\lambda\,g +\mu\,\eta\otimes\eta ,
\end{align}
see, for example, \cite{DeyRoy-2020,DSB-2024,ven-2022}.
Recall that the Hessian, ${\rm Hess}_{\,v}$, is a symmetric (0,2)-tensor defined by
\[
 {\rm Hess}_{\,v}(X,Y) = g(\nabla_X\nabla v, Y)\quad (X,Y\in\mathfrak{X}_M).
\]
For a gradient almost $\ast$-$\eta$-RS, \eqref{E-sol-eta2} reduces to
\begin{align}\label{E-sol-eta2-grad}
 \nabla_X\nabla v + \Ric^\sharp QX = \lambda\,X - (2n-1)\beta^2 QX + (\mu-\beta^2)\,\eta(X)\,\xi .
\end{align}

\begin{theorem}\label{T-005}
Let a weak $\beta$-KM $M({f},Q,\xi,\eta,g)$ with $\beta\ne0$ be a gradient almost $\ast$-$\eta$-RS \eqref{E-sol-ast-eta-grad}.
Then $\nabla v=\delta\,\xi$ for some $\delta\in\mathbb{R}$ and
$(M, g)$ is an Einstein manifold of scalar curvature ${r}=-2\,n(2\,n+1)\beta^2$.
\end{theorem}

\begin{proof}
Taking covariant derivative of \eqref{E-sol-eta2-grad} along a vector field $Y$ and using \eqref{2.3b}, yields
\begin{align}\label{E-sol-eta2-grad2}
\notag
 & \nabla_Y\nabla_X\nabla v + (\nabla_Y Q\Ric^\sharp)X = Y(\lambda)X - (2n-1)\beta^2 (\nabla_Y Q)X + Y(\mu)\,\eta(X)\,\xi \\
 & + (\mu-\beta^2)\beta\big\{ \big(g(X,Y) -\eta(X)\,\eta(Y)\big)\,\xi + \eta(X)\big(Y -\eta(Y)\,\xi\big)\big\} .
\end{align}
Applying \eqref{E-sol-eta2-grad2} in the expression of Riemannian curvature tensor, we obtain
\begin{align}\label{E-sol-eta2-R}
\notag
 R_{X,Y}\nabla v & = (\nabla_Y Q\Ric^\sharp)X - (\nabla_X Q\Ric^\sharp)Y
 + X(\lambda)Y - Y(\lambda)X + \big(X(\mu)\eta(Y) - Y(\mu)\eta(X)\big)\,\xi \\
 & + (\mu-\beta^2)\beta\{\eta(Y)X - \eta(X)Y\} + (2n-1)\beta^2 \{ (\nabla_Y Q)X - (\nabla_X Q)Y \} .
\end{align}
Moreover, an inner product of \eqref{E-sol-eta2-R} with $\xi$ and use of \eqref{E-nS-10b} and Proposition~\ref{T-lambda0}
yields
\begin{align}\label{E-sol-eta2-R-xi}
\notag
 & g(R_{X,Y}\nabla v, \xi) = g((\nabla_Y Q\Ric^\sharp)\xi, X) - g((\nabla_X Q\Ric^\sharp)\xi, Y) + X(\lambda)\eta(Y) - Y(\lambda)\eta(X) \\
\notag
 & + X(\mu)\eta(Y) - Y(\mu)\eta(X) + (2n-1)\beta^2 \{ g( (\nabla_Y Q)X, \xi) - g((\nabla_X Q)Y, \xi)\}  \\
& = X(\lambda+\mu)\eta(Y) - Y(\lambda+\mu)\eta(X) = 0 .
\end{align}
On the other hand, using \eqref{2.4} we obtain
\begin{align}\label{E-sol-eta2-R-xi2}
 g(R_{X,Y}\nabla v, \xi) = - g(R_{X,Y}\xi, \nabla v) = -\beta^2\{\eta(X)Y(v) - \eta(Y)X(v)\} .
\end{align}
Comparing \eqref{E-sol-eta2-R-xi} and \eqref{E-sol-eta2-R-xi2}
and using the assumption $\beta\ne0$, we conclude that
\begin{align}\label{E-sol-eta2-R-xi3}
 \eta(X)Y(v) - \eta(Y)X(v) = 0 .
\end{align}
Taking $Y=\xi$ in \eqref{E-sol-eta2-R-xi3}, we obtain
$X(v)=\eta(X)\,\xi(v)$. Thus, $X(v)=0\ (X\perp\xi)$, i.e., the soliton vector field $V=\nabla v$ is collinear with $\xi$.
Applying Theorem~\ref{T-004} completes the proof.
\end{proof}

\section{Conclusion}

We introduced the $\ast$-Ricci tensor for weak metric structures (defined by the author and R.\,Wolak), which are more general than the almost contact metric structure. We extended the theory of $\ast$-$\eta$-RS for weak $\beta$-KM -- a distinguished class of manifolds with a weak contact metric structure,
defined as a generalization of K.~Kenmotsu's concept.
Our results generalize the study of other authors and present conditions under which a weak $\beta$-KM equipped with a $\ast$-$\eta$-RS carries an Einstein metric. A~future task is to study the interaction of weak $\beta$-KM (as well as other weak almost contact metric manifolds)
with almost $\ast$-$\eta$-RS, that is, when $\lambda,\mu$ in \eqref{E-sol-ast-eta0} are smooth functions on the~manifold.

\end{document}